\documentclass[12pt]{article} 
\usepackage{geometry} 
\usepackage{stmaryrd}
\usepackage{textcomp}
\usepackage{enumerate}
\usepackage{setspace,latexsym,epsfig}
\usepackage{amssymb,amsthm,amsfonts}
\usepackage{amsthm}
\usepackage{amsmath}
\usepackage{graphicx}
\usepackage{marvosym}
\usepackage{empheq}
\usepackage{latexsym}
\usepackage{endnotes}
\usepackage{fontenc}
\usepackage{color}
\usepackage{hyperref}
\usepackage{amsfonts,epsfig,graphicx,multirow,rotating}
{\normalfont\normalsize\bfseries}
\newtheorem{definition}{Definition}
\newtheorem{prop}{Proposition}
\newtheorem{theorem}{Theorem}

\newtheorem{problem}{Problem}

\newtheorem{conj}{Conjecture}
\newtheorem{remark}{Remark}
\begin{document}
\date{} 
\begin{center}
\uppercase{\bf On Double 3-term Arithmetic Progressions}
\vskip 20pt
{\bf Tom Brown}\\
{\small Department of Mathematics, Simon Fraser University, Burnaby, B.C., Canada}\\
{\tt  tbrown@sfu.ca}\\ 
\vskip 10pt
{\bf Veselin Jungi\'{c}}\\
{\small Department of Mathematics, Simon Fraser University, Burnaby, B.C., Canada}\\
{\tt vjungic@sfu.ca}\\ 
\vskip 10pt
{\bf Andrew Poelstra}\\
{\small Department of Mathematics, Simon Fraser University, Burnaby, B.C., Canada}\\
{\tt asp11@sfu.ca}\\ 
\end{center}
\vskip 30pt

\centerline{\bf Abstract}
\noindent

In this note we are interested in the problem of whether or not every increasing sequence of positive integers $x_1x_2x_3\cdots$
with bounded gaps must contain a {\it double} 3-term arithmetic progression, i.e., three terms $x_i$, $x_j$, and $x_k$ such that
$i+k=2j$ and $x_i+x_k=2x_j$. We consider a few variations of the problem, discuss some related properties of double arithmetic progressions, and present several results obtained by using \texttt{RamseyScript}, a high-level scripting language.



\section{Introduction}

In 1987, Tom Brown and Allen Freedman ended their paper titled {\it Arithmetic progressions in lacunary sets} \cite{BF} with the following conjecture.

\begin{conj} Let $(x_i)_{i\geq 1}$ be a sequence of positive integers with $1 \leq x_i \leq K$. Then there are two consecutive intervals of positive integers $I,J$ of the same length, with $\sum _{i\in I}x_i = \sum_{j\in J} x_j$. Equivalently, if $a_1 <a_2 <\cdots$ satisfy $a_{n+1}-a_n \leq K$, for all $n$, then there exist $i<j<k$ such that $i+k=2j$ and $a_i+a_k =2a_j$.
\end{conj}

If true, Conjecture 1 would imply that if the sum of the reciprocals of a set $A=\{a_1<a_2<a_3<\cdots\}$ of positive integers diverges, and $a_{n+1}-a_n \rightarrow \infty$ as $n\rightarrow \infty$, and there exists $K$ such that $a_{i+1}-a_i\le a_{j+1}-a_j + K$ for all $1\le i \le j,$ then $A$ contains a 3-term arithmetic progression.  This is a special case of the  famous Erd\H{o}s conjecture that if the sum of the reciprocals of a set $A$ of positive integers diverges, then $A$ contains arbitrarily long arithmetic progressions.

Conjecture 1 is a well-known open question in combinatorics of words and it is usually stated in the following form: 
\begin{quote}
Must every infinite word on a finite alphabet consisting of positive integers contain an additive square, i.e., two adjacent blocks of the same length and the same sum?
\end{quote}
The answer is trivially yes in case the alphabet has size at most 3.  For more on this question see, for example,  \cite{Au, Freedman}, \cite{Jaroslaw}.  Also see \cite{HH}, \cite{Pirillo} and \linebreak \cite{RSZ}.

We mention two relatively recent positive results. Freedman \cite{Freedman} has shown that if $a<b<c<d$ satisfy the Sidon equation $a+d=b+c$, then every word on $\{a,b,c,d\}$ of length 61 contains an additive square. His proof is a clever reduction of the general problem to several cases that are then checked by computer. 

Ardal, Brown, Jungi\'{c}, and Sahasrabudhe \cite{ABJS} proved that if an infinite word $\omega =a_1a_2a_3\cdots$ has the property that there is a constant $M$, such that for each positive integer $n$ the number of possible sums of $n$ consecutive terms in $\omega$ does not exceed $M$, then for each positive integer $k$ there is a $k$-term arithmetic progression $\{m+id:i=0,\cdots, k-1\}$ such that
$$\sum  _{i=m+1}^{m+d}a_i=\sum  _{i=m+d+1}^{m+2d}a_i=\cdots =\sum  _{i=m+(k-2)d+1}^{m+(k-1)d}a_i.$$
The proof of this fact is based on van der Waerden's theorem  \cite{vdW}.

This note is inspired by the second statement in Conjecture 1. Before restating this part of the conjecture we introduce the following terms.

We say that an increasing sequence of positive integers $a_1,a_2,a_3,\dots$ has {\it bounded gaps} if there is a constant $K$ such that
$$a_{n+1}-a_n\leq K$$
for all positive integers $n$.

We say that an increasing sequence of  positive integers $a_1,a_2,a_3,\dots$ contains a {\it double $k$-term arithmetic progression} if there are $p_1<p_2<\cdots <p_k$ such that both $\{a_{p_1},a_{p_2}, \dots, a_{p_k}\}$ and $\{p_1,p_2, \dots, p_k\}$
are arithmetic progressions.
\begin{problem}Does every increasing sequence of positive integers with bounded gaps contain a {\it double} 3-term arithmetic progression?
\end{problem}

It is straightforward to check that Problem 1 is equivalent to the question above concerning additive squares:  Given positive integers $K$ and $a_1 < a_2 <a_3 <\cdots$,  with $a_{i+1}-a_i \le K$ for all $i \ge 1,$ let $x_i = a_{i+1} - a_i, \ i \ge 1.$  Then $x_1x_2x_3 \cdots $ is an infinite word on a finite alphabet of positive integers.  Given an infinite word $x_1x_2x_3 \cdots$ on a finite alphabet of positive integers, define $a_1, a_2, a_3, \ldots$ recursively by $a_1 \in \mathbb{N}, a_{i+1} = x_i + a_i, \ i\ge 1.$   
Then $a_1 < a_2 < a_3 <\cdots$, and $a_{i+1} - a_{i} \le K$ for some $K$ and all $i \ge 1.$  In both cases, an additive square in $x_1x_2x_3\cdots $ corresponds exactly to a double 3-term arithmetic progression in $a_1 <a_2 < a_3 <\cdots$.

The existence of an infinite word on four integers with no \textit{additive cubes}, i.e., with no three consecutive blocks of the same length and the same sum, established by Cassaigne, Currie, Schaeffer, and Shallit \cite{CCSS}, translates into the fact that there is an increasing sequence of positive integers with bounded gaps with no double 4-term arithmetic progression.

But what about a {\it double} variation on van der Waerden's theorem?

\begin{problem} If  the set of positive integers  is finitely colored, must there exist a color class, say $A = \{a_1 < a_2 < a_3 < \cdots \}$ for which there exist $i < j < k$ with
$a_i + a_k = 2a_j$  and $i + k = 2j$?
\end{problem}

We have just seen that an affirmative answer to Problem 1 gives an affirmative answer to the question concerning additive squares.  It is also true that an affirmative answer to Problem 1 implies an affirmative answer to Problem 2.

\begin{prop}
Assume that every increasing sequence of positive integers \linebreak 
$x_1x_2x_3\cdots$ with bounded gaps contains a double 3-term arithmetic progression.  Then if the set of positive integers  is finitely colored, there must exist a color class, say $A = \{a_1 < a_2 < a_3 < \cdots \}$, which contains a double 3-term arithmetic progression.

\end{prop}

\begin{proof}
We use induction on the number of colors, denoted by $r$.  For $r=1$ the conclusion trivially follows.  Now assume that for every $r$-coloring of $\mathbb{N}$ there exists a color class which contains a double 3-term arithmetic progression.  By the Compactness Principle there exists $M\in \mathbb{N}$ such that every $r$-coloring of $[1,M]$ (or of any translate of $[1,M]$) yields a monochromatic double 3-term arithmetic progression.

Assume now that there is an $(r+1)$-coloring of $\mathbb{N}$ for which there does \textit{not} exist a monochromatic double 3-term arithmetic progression.  Let the $(r+1)$st color class be $C(r+1)=\{x_1<x_2< \cdots \}.$  By the induction hypothesis on $r$ colo§rs, $C(r+1)$ is infinite.  By the assumption that every increasing sequence of positive integers $x_1x_2x_3\cdots$ with bounded gaps contains a double 3-term arithmetic progression, $C(r+1)$ does not have bounded gaps.  In particular, there is $p\ge1$ such that $x_{p+1}-x_p\ge M+2.$  But then the interval $[x_p +1, x_{p+1}-1]$ contains a translate of $[1,M]$ and is colored with only $r$ colors, so that $[x_p +1, x_{p+1}-1]$ does contain a monochromatic double 3-term arithmetic progression.  This contradiction completes the proof.
\end{proof}

More generally, if the set of positive integers is finitely colored  and if each color class is regarded as an increasing sequence, must there be a monochromatic double $k$-term arithmetic progression, for a given positive integer $k$? What if the gaps between consecutive elements colored with same color are pre-prescribed, say at most 4 for the first color, at most 6 for the second color, and at most 8 for the third color, and so on? 

In the spirit of van der Waerden's numbers $w(r, k)$ \cite{GRS} we define the following.

\begin{definition}
For given positive integers $r$ and $k$ greater than 1, let $w^\ast(r,k)$ be the least integer, if it exists, such that for any $r$-coloring of the interval $[1,w^\ast(r,k)]$ there is a monochromatic double $k$-term arithmetic progression.

For given positive numbers $r$, $k$, $a_1,a_2,\ldots ,a_r$ let $w^\ast(k;a_1,a_2,\ldots, a_r)$ be the least integer, if it exists, such that for any $r$-coloring of the interval \linebreak $[1,w^\ast(k;a_1,a_2,\ldots, a_r)]=A_1\cup A_2\cup \cdots \cup A_r$  such that for each $i$ the gap between any two consecutive elements in $A_i$ is not greater than $a_i$, there is a monochromatic double $k$-term arithmetic progression.
\end{definition} 

We will show that $w^\ast(2,3)$ is relatively simple to obtain. We will give lower bounds for $w^\ast(3,3)$ and $w^\ast(4,2)$ and a table with values of $w^\ast(3;a_1,a_2,a_3)$ for various triples $(a_1,a_2,a_3)$ and propose a related conjecture.

We will share with the reader some insights related to the general question about the existence of double 3-term arithmetic progressions in increasing sequences with bounded gaps.

Finally, we will describe  \texttt{RamseyScript}, a high-level scripting language  developed by the third author that was used to obtain the colorings and bounds that we have established.

\section{$w^\ast(r,3)$}

Now we look more closely at $w^\ast(r,3),$ the least integer, if it exists, such that for every $r$-coloring of the interval $[1,w^\ast(r,3)]$ there is a monochromatic double 3-term arithmetic progression.  
 
Suppose that $w^\ast(r,3)$ does not exist for some $r$, but $w^\ast(r-1,3)$ does exist.  Then, by the Compactness Principle, there is a coloring of the positive integers with $r$ colours, say with colour classes $A_1, A_2,  \ldots , A_r$, such that no colour class contains a double 3-term arithmetic progression.  Then (a) $A_1$ contains no double 3-term arithmetic progression, (b) $A_1$ has bounded gaps because $w^\ast(r-1,3)$  exists, and (c) $A_1$ is infinite, because $w^\ast(r-1,3)$  exists.

Let $d_1, d_2,\ldots$  be the sequence of consecutive differences of the sequence $A_1$.  That is, if $A_1 = \{a_1, a_2, a_3, \ldots\}$ then $d_n = a_n - a_{n-1}$, $n \geq 1$.  Then the sequence $d_1, d_2, \ldots$ is a sequence on a finite set of integers which does not contain any additive square.

Thus if there exists $r$ such that $w^\ast(r,3)$ does not exist,  then there exists a sequence on a finite set of integers which does not contain an additive square.

It is conceivable that proving that $w^\ast(r,3)$ does not exist for all $r$ (if this is true!) is easier than directly proving the existence of a  sequence on a finite set of integers with no additive square.

\begin{theorem} $w^\ast (2,3) = 17$.
\end{theorem}
\begin{proof} Color $[1, m]$ with two colors, with no monochromatic double 3-term arithmetic progressions.  Then the first color class must have gaps of either 1, 2, or 3. Thus the sequence of gaps of the first color class is a sequence of 1s, 2s, and 3s, and this sequence must have length at most 7, otherwise there is an additive square, which would give a double 3-term arithmetic progression in the first color class.  Hence, the first colour class can contain at most 8 elements (only 7 consecutive differences) and similarly for the second colour class.  This shows that $w^\ast(2,3) \leq 8+8+1=17$.  On the other hand, the following 2-coloring of $[1,16]$ has no monochromatic double 3-term arithmetic progression:
$$0010110100101101.$$
Hence $w^\ast(2,3)=17$.
\end{proof}

\begin{theorem} $w^\ast (3,3) \geq 414$.
\end{theorem}

The following 3-coloring of $[1,413]$ avoids monochromatic 3-term double arithmetic progressions:
$$\begin{array}{l}
0101102210100201200100221221010010220010112011211202210112122112202210\\
0110010220201122022002202001012212112122001001120121100110020022002110\\
2001101001121120210020011210201121122112122010110100110102201220201221\\
1210021122112122112200110011212200202202001212212112212200110010110012\\
0211212200220100112202200220200122102212211211002101220022001001100221\\
211010010110020022110010110010221211020220200220221001122011211.\\
\end{array}
$$

This coloring is the result of about 8 trillion iterations of  \texttt{RamseyScript}, using the Western Canada Research Grid\footnote{http://www.westgrid.ca}. We started with a seed 3-coloring of the interval $[1, 61]$ and searched the entire space of extensions. Figure \ref{A1} gives the number of double 3-AP free extensions of the seed coloring versus their lengths. 

\begin{figure}[h] 
   \centering
   \includegraphics[width=12cm]{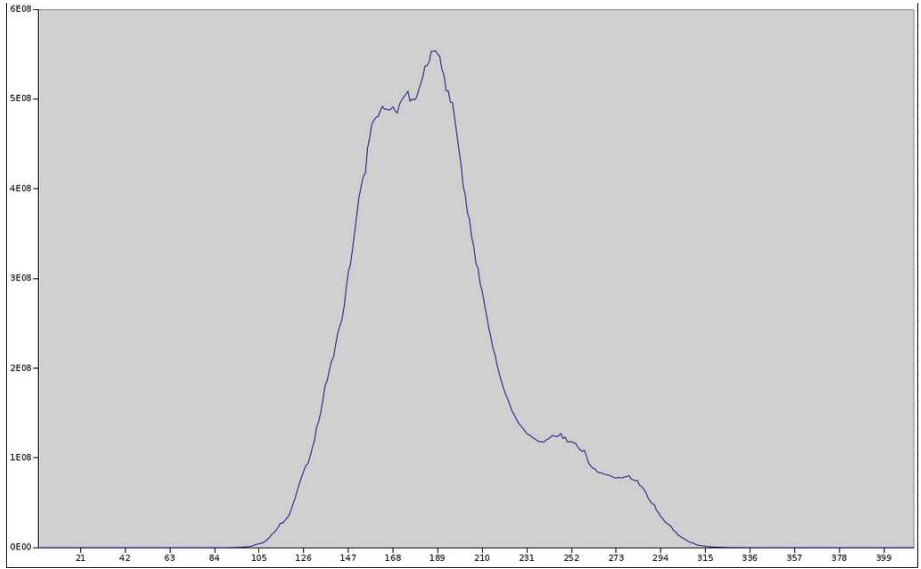} 
   \caption{Number of double 3-AP free extensions versus length}
   \label{A1}
\end{figure}

To get more information about $w^\ast (3,3)$ we define $w^\ast (3,3;d)$ to be the smallest $m$ such that whenever $[1, m]$ is 3-coloured so that each colour class has maximum gap at most $d$, then there is a monochromatic double 3-term arithmetic progression.  Our goal was to compute  $w^\ast(3;3;d)$ for small values of $d$.  (See Table 1.)

\begin{table}[h]
\label{table:w(3,3,d)}
\begin{center}
 \begin{tabular}{|cc|c|}
  \hline
 & & $w^\ast(3,3;d)$	\\
    \hline
  \multirow{6}{*}{\begin{sideways}Max gap $d$\end{sideways}}
    & 2   & 11	\\
    & 3   & 22	\\
    & 4   & 39\\
    & 5  & 100\\
     & 6  & $>152$	\\
    & 7  & ?	\\
  \hline
  \end{tabular}
\end{center}
\caption{Known Values of $w^\ast(3,3;d)$}
\end{table}

We note that $w^\ast (3,3;d)$ is already difficult to compute when $d$ is much smaller than $w^\ast(2,3)=17$.  (In a 3-coloring containing no monochromatic double 3-term arithmetic progression the maximum gap size of any color class is 17.)

Freedman \cite{Freedman} showed that there were 16 double 3-AP free 51-term sequences having the maximum gap of at most $4$. The fact that $w^\ast(3,3;4)=39$ is an interesting contrast, and shows that considering a single sequence instead of partitioning an interval of positive integers into three sequences is somewhat less restrictive. 

\begin{theorem} $w^\ast (2,4) \geq 30830$.
\end{theorem}

Starting with the seed 2-coloring $[1,10]=\{1, 4, 6, 7\}\cup\{2, 3, 5, 8, 9, 10\}$, after $2\cdot10^8$ iterations \texttt{RamseyScript} produced a double 4-AP free 2-coloring of the interval $[1,30829]$ that is available at the web page \href{http://people.math.sfu.ca/~vjungic/Double/w-4-2.txt}{people.math.sfu.ca/~vjungic/Double/w-4-2.txt}.

\section{$w^\ast(3;a,b,c)$ and $w^\ast(k;a,b)$}

Recall that $w^\ast(3; a, b, c)$ is the least number such that every
3-coloring of\linebreak$[1, w^\ast(3; a, b, c)]$, with gap sizes on the three colors restricted
to $a$, $b$, and $c$, respectively, has a monochromatic double 3-term arithmetic progression. Similarly, $w^\ast(k; a, b)$ is the least number such that every 2-coloring of $[1, w^\ast(k; a, b)]$, with gap sizes on the two colors restricted
to $a$ and $b$, respectively, has a monochromatic double $k$-term arithmetic progression.

Table 2 shows values of $w^\ast(3;a,b,c)$ for some small values of $a$, $b$, and $c$. Table 3 shows values of $w^\ast(k;a,b)$ for some small values of $a$, $b$, and $k$.

\begin{table}[h]
\parbox{.45\linewidth}{
\centering
 \begin{tabular}{|cc|ccccc|}
  \hline
  & & \multicolumn{5}{|c|}{Max Green Gaps}	\\
  & & 3 & 4 & 5 & 6 & 7+	\\
  \hline
  \multirow{6}{*}{\begin{sideways}Max Blue Gaps\end{sideways}}
    & 3  & 22 &    &    &    &	\\
    & 4  & 31 & 31 &    &    &	\\
    & 5  & 33 & 38 & 43 &    &	\\
    & 6  & 33 & 41 & 44 & 45 &	\\
    & 7  & 33 & 41 & 46	& 46 & 46	\\
    & 8+ & 33 & 41 & 46	& 46 & 47	\\
  \hline
  \multicolumn{7}{c}{Max Red Gap 3}
  \end{tabular}
}
\hfill
\parbox{.45\linewidth}{
 \begin{center}
  \begin{tabular}{|cc|cccc|}
  \hline
  & & \multicolumn{4}{|c|}{Max Green}	\\
  & & 5 & 6 & 7 & 8+ 	\\
  \hline
  \multirow{4}{*}{\begin{sideways}Max Blue\end{sideways}}
    & 5   & 100    &        &   &	\\
    & 6   & $>113$ & $>133$ &   &	\\
    & 7   & ?      & ?      & ?	&	\\
    & 8+  & ?      & ?      & ? & ?	\\
  \hline
  \multicolumn{6}{c}{Max Red Gap 5}
  \end{tabular}
  \end{center}
}

\parbox{.45\linewidth}{
\begin{center}
 \begin{tabular}{|cc|cccccc|}
  \hline
  & & \multicolumn{6}{|c|}{Max Green Gaps}	\\
  & & 4 & 5 & 6 & 7 & 8 & 9+	\\
  \hline
  \multirow{8}{*}{\begin{sideways}Max Blue Gaps\end{sideways}}
    & 4   & 39 &    &        &        &        &	\\
    & 5   & 49 & 63 &        &        &        &	\\
    & 6   & 56 & 79 & 91     &        &        &	\\
    & 7   & 76 & 96 & $>$105 & $>$121 &        & 	\\
    & 8   & 81 & 96 & $>$114 & $>$131 & $>$131 &	\\
    & 9   & 81 & 96 & $>$114 & $>$133 & $>$133 & $>$133	\\
    & 10  & 81 & 96 & $>$114 & $>$133 & $>$135 & $>$135	\\
    & 11+ & 81 & 97 & $>$114 & $>$133 & $>$135 & $>$135	\\
  \hline
  \multicolumn{8}{c}{Max Red Gap 4}
  \end{tabular}

\end{center}
}
\caption{Known Values and Bounds for $w^\ast(3;a,b,c)$}
\end{table}

\begin{table}[h]
\parbox{.45\linewidth}{
 \begin{center}
  \begin{tabular}{|cc|cc|}
  \hline
  & & \multicolumn{2}{|c|}{Red}	\\
  & & 2 & 3 	\\
  \hline
  \multirow{2}{*}{\begin{sideways}Blue\end{sideways}}
    & 2  & 7  &	\\
    & 3  & 11 & 17	\\
  \hline
  \multicolumn{4}{c}{Double 3-AP's}
  \end{tabular}
  \end{center}
}
\hfill
\parbox{.45\linewidth}{
 \begin{center}
  \begin{tabular}{|cc|ccc|}
  \hline
  & & \multicolumn{3}{|c|}{Red}	\\
  & & 2 & 3 & 4+ 	\\
  \hline
  \multirow{3}{*}{\begin{sideways}Blue\end{sideways}}
    & 2   & 11 &         &	\\
    & 3   & 22 & $>176$  &	\\
    & 4+  & 22 & $>2690$ & $>3573$	\\
  \hline
  \multicolumn{5}{c}{Double 4-AP's}
  \end{tabular}
  \end{center}
}
\parbox{.45\linewidth}{
\begin{center}
  \begin{tabular}{|cc|cccc|}
  \hline
  & & \multicolumn{4}{|c|}{Red}	\\
  & & 2 & 3 & 4 & 5+ 	\\
  \hline
  \multirow{3}{*}{\begin{sideways}Blue\end{sideways}}
    & 2   & 15 &         &  &	\\
    & 3   & 37 & $>131000$  &  &	\\
    & 4   & $>25503$ & ? & ? &	\\
    & 5+  & $>33366$ & ? & ? & ? 	\\
  \hline
  \multicolumn{6}{c}{Double 5-AP's}
  \end{tabular}

  \end{center}
}
\caption{Known Values and Bounds for $w^\ast(k;a,b)$}
\end{table}

Based on this evidence, we propose the following conjecture.

 \begin{conj} 
The number $w^*(3, 3)$ exists. The number  $w^*(2, 4)$ does not exist.
\end{conj}

Our guess would be that $w^*(3, 3)<500$. Also we recall that $w^*(2, 3)=17$ and $w^\ast (2,4) \geq 30830$.

\section{Double 3-term Arithmetic Progressions in Increasing Sequences of Positive Integers}

In this section, we return to Problem 1: the existence of double 3-term arithmetic progressions in infinite sequences of positive integers with bounded gaps.

We remind the reader of the meaning of the following terms from combinatorics of words.
 
 An infinite word on a finite subset $S$ of $\mathbb{Z}$, called the \textit{alphabet}, is defined as  a map $\omega:\mathbb{N}\to S$ and is usually written as $\omega=x_1x_2\cdots,$ with $x_i\in S$, $i\in \mathbb{N}$. For $n\in \mathbb{N}$, a \textit{factor} $B$ of the infinite word $\omega$ of length $n=|B|$  is the image of a set of $n$ consecutive positive integers by $\omega$, $B=\omega(\{ i,i+1,\cdots, i+n-1\})=x_ix_{i+1}\cdots x_{i+n-1}$. The \textit{sum} of the factor $B$ is $\sum B=x_i+x_{i+1}+\cdots +x_{i+n-1}$. A factor $B=\omega(\{ 1,2,\cdots, n\})=x_1x_{2}\cdots x_{n}$ is called a \textit{prefix} of $\omega$.

\begin{theorem}
The following statements are equivalent:

\begin{itemize}

\item[(1)]  For all $k>1$, every infinite word on $\{1, 2, \cdots , k\}$ has two adjacent factors with equal length and equal sum.

\item[(1a)]  For all $k>1$, there exists $R = R(k)$ such that every word on $\{1, 2, \cdots , k\}$ of length $R$ has two adjacent factors with equal length and equal sum. 

\item[(2)]  For all $n > 1$, if $x_1  <  x_2  < x_3  <\cdots$ is an infinite sequence of positive integers such that
$x_{i+1} - x_i  \leq  n$ for all $i > 1$, then there exist $1 \leq i < j < k$ such that $x_i + x_k = 2x_j$ and $i + k = 2j$. 
\item[(2a)]  For all $n>1$, there exists $S = S(n)$ such that if $x_1  <  x_2  < x_3  < \cdots  <  x_S$ are positive integers with $x_{i+1} - x_i  \leq  n$ whenever $1  \leq  i  \leq S - 1$, then there exist $1 \leq i < j < k \leq S$  such that $x_i + x_k = 2x_j$ and $i + k = 2j$.

\item[(3)] For all $t > 1$, if $\mathbb{N} = A_1 \cup A_2 \cup \cdots \cup A_t$, then there exists $q$, $1\leq q \leq t$, such that if $A_q = \{x_1  < x_2  <  \cdots\}$, there are  $1 \leq i < j < k$ such that  $x_i + x_k = 2x_j$ and $i + k = 2j$.

\item[(3a)]  For all $t > 1$, there exists $T = T(t)$ such that for all $a > 1$, if $\{a, a + 1, \cdots , a + T - 1\} = A_1 \cup A_2 \cup \cdots \cup A_t$, then there exists $q$, $1 \leq q \leq t$, such that if
$A_q = \{x_1  < x_2  < \cdots < x_p\}$, there are  $1 \leq i < j < k$ such that $x_1 + x_k = 2x_j$ and $i + k = 2j$.
\end{itemize}
\end{theorem}

\begin{remark} Note that in (3) and (3a) the statements concern coverings (by not necessarily disjoint sets) and not partitions (colorings).  This turns out to be essential, since if we used colorings in (3) and (3a) (call these new statements (3') and (3a')), then (3') would not imply (2), although (2) would still imply (3a').  This can be seen from the proofs below.
\end{remark}
\begin{remark}  In each case $i = 1,2,3$, the statement {\it (ia)} is the finite form of the statement {\it (i)}.
\end{remark}
\begin{proof}
We start by proving that (2) implies (2a).  (The proof that (1) implies (1a) follows the same form, and is a little more routine.)

Suppose that (2a) is false.  Then there exists $n$ such that for all $S > 1$ there are $x_1  <  x_2  < x_3  < \cdots  <  x_S$, with $x_{i+1} - x_i  \leq  n$ whenever $1  \leq  i  \leq  S - 1$, such that there do not exist $1 \leq i < j < k \leq S$  such that $x_i + x_k = 2x_j$ and $i + k = 2j$.  Replace $x_1  <  x_2  < x_3  < \cdots  <  x_S$  by its characteristic binary word (of length $x_S$)
$$B_S = b_1 b_2 b_3 \cdots  b_{x_S}$$ 
defined by  $b_i = 1$ if $i$ is in $\{x_1 ,  x_2 ,  x_3,   \ldots  ,  x_S\}$, and $b_i = 0$ otherwise.
Let $H$ be the (infinite) collection of binary words obtained in this way.  Note that if $B_S$ is in $H$, then consecutive 1s in $B_S$ are separated by at most $n-1$ 0s.

Now construct, inductively, an infinite binary word $w$ such that each prefix of $w$ is a prefix of infinitely many words $B_S$ in $H$ in the following way.  Let $w_1$ be a prefix of an infinite set $H_1$ of words in $H$.  Let $w_1 w_2$ be a prefix of an infinite set $H_2$ of words in $H_1$.  And so on.  Set $w = w_1 w_2  \cdots $ .
        
 Define $x_1  <  x_2  < x_3  < \cdots$ so that $w$ is the characteristic word of $x_1  <  x_2  < x_3  < \cdots $ and note that  $x_{i+1} - x_i  \leq  n$ for all $i> 1$.  Now it  follows that there cannot exist $1 \leq i < j < k$ with $x_1 + x_k = 2x_j$ and $i + k = 2j$.  (For these $i, j, k$ would occur inside some prefix of $w$.  But that prefix is itself a prefix of some word
$B_S = b_1 b_2 b_3 \cdots  b_S$, where there do not exist such $i, j, k$.) Thus if (2a) is false, (2) is false.

Next we prove that (3) implies (3a).  Suppose that (3a) is false.  Then there exists $t$ such that for all $T$ there is, without loss of generality, a covering $\{1, 2, \ldots , T \} = A_1 \cup A_2 \cup \cdots \cup A_t$ such that there does not exist $q$ with $A_q = \{x_1  < x_2  <  \cdots < x_p\}$ and $i < j < k$ with $x_1 + x_k = 2x_j$ and $i + k = 2j$.  Represent the cover $\{1, 2, \ldots , T \} = A_1 \cup A_2 \cup \cdots \cup A_t$ by a word $B_T = b_1 b_2 b_3 \cdots  b_T$ on the alphabet consisting of the non-empty subsets of $\{1,2,\ldots ,t\}$.  Here for each $i$, $1 \leq i \leq T$, $b_i = \{ \mbox{the set of }p, 1 \leq p \leq t, \mbox{ such that } i \mbox{ is in }A_p\}$.  Let $H$ be the set of all words $B_T$ obtained in this way.  Construct an infinite word $w=w_1w_2w_3\dots$ on  the alphabet consisting of the non-empty subsets of $\{1,2,\ldots,t\}$, such that each prefix of $w$ is a prefix of infinitely many of the words $B_T$ in $H$.  Thus $w$ represents a cover $\mathbb{N} = A_1 \cup A_2 \cup \cdots \cup A_t$, where $A_i = \{j \ge 1 \mbox{ such that $i$ is in }w_j\}$, $1 \leq i \leq t$, for which there does not exist $i$, $A_i = \{x_1  < x_2  <  \cdots \}$, with  $1 \leq i < j < k$ such that $x_1 + x_k = 2x_j$ and $i + k = 2j$, contradicting (3).

It is not difficult to show  that (1) is equivalent to (2), that (1) is equivalent to (1a), that (2a) implies (2), and that (3a) implies (3).  We have shown that (2) implies (2a) and that (3) implies (3a).  

The final steps are:

Proof that (3) implies (2).  If $n$ and $A_0 = \{x_1  <  x_2  < x_3  < \cdots  \}$ are given, with $x_{i+1} - x_i  \leq  n$ for all $i > 1$, let $A_i = A_0 + i$,  $0 \leq i \leq n - 1$.  Then $\mathbb{N} = A_0 \cup A_1 \cup \cdots \cup A_{n-1}$, and now (3) implies (2).

Proof that (2) implies (3) and (3a).  Assume (2), and use induction on $t $ to show that (3) and (3a) are true for $t$. (Note that (3) for a given value of $t$ is equivalent to (3a) for the same value of $t$.)  For $t = 1$ this is trivial.    Fix $t \ge 1,$ assume (3) and (3a) for this $t$, and let  $\mathbb{N} = A_1 \cup A_2 \cup \cdots \cup A_{t+1}$.
If $A_{t+1}$ is finite we are done by the induction hypothesis on (3). If $A_{t+1}$ has bounded gaps, we are done by (2).  In the remaining case, there are arbitrarily long intervals which are subsets of  $A_1 \cup A_2 \cup \cdots \cup A_t$, and we are done by the induction hypothesis on (3a).

\end{proof}

\begin{remark} If true, perhaps (3a) can be proved by a method such as van der Waerden's proof that any finite coloring of $\mathbb{N}$ has a monochromatic 3-AP.
\end{remark}

Here is another remark on double 3-term arithmetic progressions.

\begin{theorem} The following two statements are equivalent:
\begin{itemize}
\item[(1)]  For all $n\geq 1$, every infinite sequence of positive integers  $x_1<x_2<\cdots$  such that $x_{i+1}-x_i\leq n$ contains a double 3-term arithmetic progression.
\item[(2)]  For all $n\geq 1$, every infinite sequence of positive integers  $x_1<x_2<\cdots$  such that $x_{i+1}-x_i\leq n$ contains a double 3-term arithmetic progression $x_i,x_j,x_k$ with the property that $j-i=k-j\geq m$ for any fixed $m\in \mathbb{N}$.
 \end{itemize}
 \end{theorem}
 
\begin{proof} Certainly (2) implies (1). We prove that (1) implies (2).

Let $n$ and $m$ be given positive integers.  Let $X = \{x_1 < x_2 < \cdots \}$ be an infinite sequence with gaps from $\{1,\ldots, n\}$. 
For $j\in \mathbb{N}$ we define $y_j=x_{jm+1}-x_{(j-1)m+1}$. Note that $m \leq y_j \leq nm$. Next we define an increasing sequence $Z=\{ z_1~<~z_2~<~\cdots\}$ with gaps from $\{m,m+1, \ldots, nm\}$ by
$$z_i=\sum _{j=1}^iy_j= \sum _{j=1}^ix_{jm+1}- \sum _{j=0}^{i-1}x_{jm+1}.$$
By (1) the sequence $Z$ contains a double 3-term arithmetic progression $z_p,z_q, z_r$ with
$$z_r-z_q=z_q-z_p \mbox{ and } p+r=2q.$$
It follows that 
$$\sum _{j=q+1}^rx_{jm+1}- \sum _{j=q}^{r-1}x_{jm+1}=\sum _{j=p+1}^qx_{jm+1}- \sum _{j=p}^{q-1}x_{jm+1}$$
and
$$x_{rm+1}-x_{qm+1}=x_{qm+1}-x_{pm+1}.$$
From
$$(pm+1)+(rm+1)=m(p+r)+2=2mq+2=2(mq+1)$$
we conclude that $x_{pm+1}, x_{qm+1},x_{rm+1}$ form a double 3-term arithmetic progression with 
$$rm+1-(qm+1)=(r-q)m\geq m.$$

Since $m$ and $X$ are arbitrary, we conclude that (2) holds.

\end{proof}

We wonder if one could get some intuitive ``evidence''  that it is easier to show that $w^\ast(3,3)$ exists than it is to show that every increasing sequence with gaps from  $\{1,2,3, \ldots , 17\}$ has a double 3-term arithmetic progression.  The ``17'' is chosen because in a 3-coloring of $[1, m]$ which has no monochromatic double 3-AP, the gaps between elements of this color class are colored with 2 colors, and $w^\ast(2,3) = 17$.

\texttt{RamseyScript} was used for search of an increasing sequence with gaps from  $\{1,2,3, \ldots , 17\}$ with no double 3-term arithmetic progressions. The first search produced a sequence of the length 2207. The histogram with the distribution of gaps in this sequence  is given on Figure \ref{O1}.

\begin{figure}[h] 
   \centering
   \includegraphics[width=12cm]{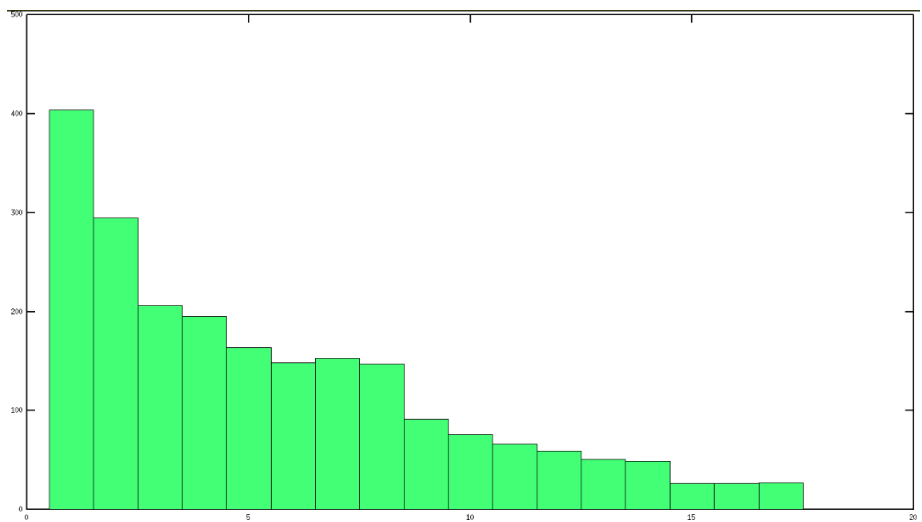} 
   \caption{Histogram of Gaps in a 2207-term Double 3-AP Free Sequence}
   \label{O1}
\end{figure}

In another attempt we changed the order of gaps in the search, taking 
$$[ 16, 12, 11, 17, 10, 14, 15, 8, 5, 3, 6, 4, 2, 1, 13, 7, 9 ]$$ instead of $[1, 2,\cdots ,17]$. \texttt{RamseyScript} produced a 5234-term double 3-AP free sequence. The corresponding histogram of gaps is given on Figure \ref{O2}.

\begin{figure}[h] 
   \centering
   \includegraphics[width=12cm]{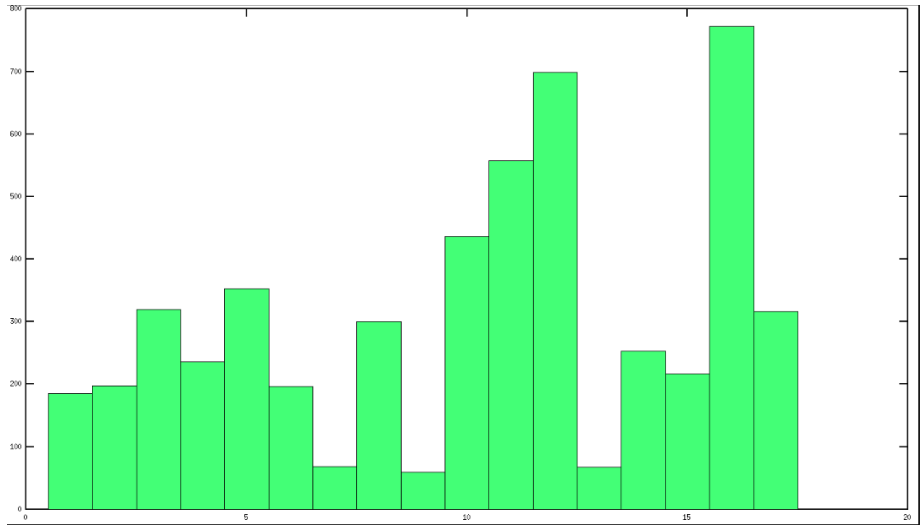} 
   \caption{Histogram of Gaps in a 5234-term Double 3-AP Free Sequence}
   \label{O2}
\end{figure}
Here are a few conclusion that one can make from this experiment.
\begin{enumerate}
\item Initial choices of the order of gaps matter very much when constructing a double 3-AP free sequence,
     because we cannot backtrack in a reasonable (human) timespan at
     these lengths.

\item We do not really know anything about how long a sequence there
     will be. 

\item The search space is very big. Table 4 gives
      the recursion tree size vs. maximum sequence length considered.
 
   \begin{table}[h]
\label{table:tree}  
\begin{center}
     \begin{tabular}{c|c}
     Max. Seq. Length&Size of Search\\
     \hline
     0       &       1\\
               1          &        18\\
               2          &        307\\
               3          &       4931\\
               4          &      78915\\
               5         &   1216147\\
               6         &  18695275\\
               7        & 278661995\\
               8          &       ???? \\
\end{tabular}
\end{center}
\caption{Recursion Tree Size vs. Maximum Sequence Length}
\end{table}
\end{enumerate}

\section{\texttt{RamseyScript}}

To handle the volume and variety of computation required by this project
and related ones, we use the utility \texttt{RamseyScript}, developed by the
third author, which provides a high-level scripting language. In creating
\texttt{RamseyScript}, we had two goals:
\begin{itemize}
\item[-] To provide a uniform framework for Ramsey-type computational
problems (which despite being minor variations of each other, are
traditionally handled by \emph{ad hoc} academic code).

\item[-] To provide a correct and efficient means to actually carry out
these computations.
\end{itemize}

To achieve these goals, \texttt{RamseyScript} appears to the user as a
declarative scripting language which is used to define a backtracking
algorithm to be run. It exposes three main abstractions: search space,
filters and targets. 

The \emph{search space} is a set of objects --- typically $r$-colorings
of the natural numbers or sequences of positive integers --- which can
be generated recursively and checked to satisfy certain conditions,
such as being squarefree or containing no monochromatic progressions.

The conditions to be checked are specified as \emph{filters}. Typically
when extending \texttt{RamseyScript} to handle a new type of problem,
only a new filter needs to be written. This saves development time and
effort compared to writing a new program, while also making available
additional features, e.g. for splitting the problem across a computing
cluster.

Finally, \emph{targets} describe the information that should be shown
to the user. The default target, \texttt{max-length}, informs the user
of the largest object in the search space which passed the filters.

With these parameters set, \texttt{RamseyScript} then runs a standard
backtracking algorithm, which essentially runs as follows:
\begin{enumerate}
\item Start with some element $x$ in the search space. For example,
$x$ might be the trivial coloring of the empty interval.
\item\label{algo:recurse} Check that $x$ passes each filter. If not,
skip steps \ref{algo:1} and \ref{algo:2}.
\item\label{algo:1} Check each target against $x$ (e.g., is $x$ the
longest coloring obtained so far?).
\item\label{algo:2} For each possible extension $\hat{x}$ of $x$,
repeat step \ref{algo:recurse}. For example, if $x$ is the interval
$[1,n]$ and the search space is the set of $r$-colorings, then the
possible extensions of $x$ are the $r$ colorings of $[1,n+1]$ which
match $x$ on the first $n$ elements.
\item Output the current state of all targets.
\end{enumerate}

Here is an example script to demonstrate these ideas and syntax:
\begin{verbatim}
# Output a brief description
echo Find the longest interval [1, n] that cannot be 4-colored
echo without a monochromatic 3-AP or a rainbow 4-AP.

# Set up environment
set n-colors 4
set ap-length 3

# Choose filters
filter no-n-aps
filter no-rainbow-aps

# Use the default target (max-length)

# Backtrack on the space of 4-colorings
search colorings
\end{verbatim}

Its output is
\begin{verbatim}
find the longest interval [1, n] that cannot be 4_colored
without a monochromatic 3_ap or a rainbow 4_ap.
Added filter ``no-3-aps''.
Added filter ``no-rainbow-aps''.
#### Starting coloring search ####
  Targets: 	max-length 
  Filters: 	no-rainbow-aps no-3-aps 
  Dump data: 	
  Seed:		[[] [] [] []]
Max. coloring (len    56): [[removed due to length]]
Time taken: 7s. Iterations: 4546107
#### Done. ####
\end{verbatim}

\texttt{RamseyScript} has many options to control the backtracking
algorithm and its output. For full details see the \texttt{README},
available alongside its source code at
\texttt{https://www.github.com/apoelstra/RamseyScript}. It is licensed under
the Creative Commons 0 public domain dedication license.

 \vspace{1cm}
 
 \textbf{Acknowledgement.} The authors would like to acknowledge the IRMACS Centre at Simon Fraser University for its support.

\end{document}